\documentclass[12pt]{article}
\usepackage{amsfonts}
\usepackage{amssymb}
\usepackage{amsmath}
\usepackage{cite}
\newtheorem{theorem}{Theorem}

\newtheorem{corollary}[theorem]{Corollary}

\newtheorem{definition}[theorem]{Definition}
\newtheorem{example}[theorem]{Example}

\newtheorem{lemma}[theorem]{Lemma}

\newenvironment{proof}[1][Proof]{\noindent\textbf{#1.} }{\ \rule{0.5em}{0.5em}}

\newcommand\subdir{\setbox0=\hbox{$\mathchar "0362$}%
\mathop{\displaystyle\smash{\raise 0.3\ht0\hbox{\hbox to \wd0%
{\hfil\vrule height4pt width0.4pt\hfil}\kern-\wd0%
\lower\ht0\box0}}}\limits}

\makeatletter
\def\mkdirsemi#1#2#3#4{\mathbin{\times\hbox to #4{}\hbox{\vrule
 \@width #1\@height #2\@depth #3}\hbox to .2em{} }}
\makeatother

\newcommand{\rnk}[2]{\lower 0.2ex\hbox{$#1$}\kern -.1em{\setminus}
                    \kern -.1em\raise 0.2ex\hbox{$#2$}}
\newcommand{\dnk}[3]{\lower 0.2ex\hbox{$#1$}\kern -.1em{\setminus}
                    \kern -.1em\raise 0.2ex\hbox{$#2$}
                    \kern -.1em{/}\kern -.1em\lower 0.2ex\hbox{$#3$}}
\newcommand{\lnk}[2]{\raise 0.2ex\hbox{$#1$}
                    \kern -.1em{/}\kern -.1em\lower 0.2ex\hbox{$#2$}}
\begin{document}

\tableofcontents

\author{Kenneth M Monks \\ Colorado State University}
\date{February 18, 2010}
\title{The M\"obius Number of the Socle of any Group}

\maketitle

\section{Definitions and Statement of Main Result}

The \textbf{incidence algebra} of a poset $P,\ $written $I\left( P\right) ,$
is the set of all real-valued functions on $P\times P$ that vanish for
ordered pairs $\left( x,y\right) $ with $x\not\leq $ $y.$  If $P$ is finite, by appropriately labeling the rows and columns of a matrix with the elements of $P$, we can see the elements of $I\left( P\right)$ as upper-triangular matrices with zeroes in certain locations.   One can prove $I\left( P\right)$ is a subalgebra of the matrix algebra (see for example \cite{Stanley}).  Notice a function $f \in I(P)$ is invertible if and only if $f\left(x,x\right)$ is nonzero for all $x \in P$, since then we have a corresponding matrix of full rank.  A natural function to consider that satisfies this property is the \textbf{incidence function} $\zeta _{P}$, the characteristic function of the relation $\leq_{P}.$  Clearly $\zeta _{P}$ is invertible by the above criterion, since $x \leq x$ for all $x \in P$.

We define the \textbf{M\"{o}bius function} $\mu _{P}$ to be the
multiplicative inverse of $\zeta _{P}$ in $I\left( P\right).$  If $P$ is a poset that has a minimum element $0$ and a maximum element $1$, the \textbf{M\"{o}bius number} of $P$ is $\mu_P(0,1)$.  Often $\mu_P(0,1)$ is abbreviated to $\mu(P)$.

Notice that if $L_n$ is the poset of divisors of a positive integer $n$ ordered by divisibility, then $\mu(n)=\mu(L_n)$ where $\mu(n)$ is the classical number-theoretic M\"{o}bius function.  Thus this is a generalization of that highly useful function from number theory to other areas of mathematics -to any situation where a poset appears!

In finite group theory, a very natural poset to study is the subgroup lattice of a finite group.  Define the \textbf{M\"{o}bius number of a group} $G$, written $\mu(G)$, to be the M\"{o}bius number of its lattice of subgroups.  A particularly important subgroup of any group is its \textbf{socle}, the subgroup generated by all minimal normal subgroups.  For instance in permutation group theory, the socle is the fundamental tool which allows classification of all primitive actions in the O'Nan-Scott Theorem \cite[Thm 4.1A]{Dixon}.  In \cite{shareshian}, Shareshian computes the M\"{o}bius numbers of several infinite families of symmetric groups by reducing to the primitive groups in each degree.  Since primitive groups are classified by socle type, the socle is particularly important for such computations.  It can be shown that the socle of any finite group is a direct product of simple groups \cite[Thm 4.3A]{Dixon}.  As our main result, we give an explicit formula for the M\"{o}bius number of the socle of a group in terms of the M\"{o}bius numbers of the simple groups which make up the socle.

\begin{theorem}\label{main}
Let $G$ be a direct product of simple groups.  More specifically, let $U_1, U_2, \ldots, U_m$ be distinct nonabelian simple groups and let $C_{p_1}, C_{p_2}, \ldots, C_{p_n}$ be abelian simple groups (cyclic groups) for some distinct primes $p_1,p_2,\ldots,p_n$.  Let $A_i$ be the size of the automorphism group of $U_i$ for each $i\in \{1\ldots m\}$.  Let $F=\overset{n}{\underset{i=1}{\sum}}{f_i}$.  Let $$G=C_{p_1}^{f_1}\times C_{p_2}^{f_2} \times \cdots \times C_{p_n}^{f_n} \times U_1^{e_1}\times U_2^{e_2} \times \cdots \times U_m^{e_m}.$$  Then 
$$\mu(G)=(-1)^{F} \overset{n}{\underset{i=1}{\prod{}}}p_i^{\binom{f_i}{2}} \overset{m}{\underset{i=1}{\prod{}}}\overset{e_i}{\underset{j=1}{\prod{}}}(\mu(U_i)-(j-1)A_i). $$
\end{theorem}

\section{Proof of Main Result}

\subsection{Complements}

\begin{definition} 
Let $x$ be an element of a lattice $L$ with minimal element 0 and maximal element 1.  An element $y \in L$ is a {\bf complement} for $x$ in $L$ if and only if the greatest lower bound of $x$ and $y$ is 0 and the least upper bound of $x$ and $y$ is 1.  We often write $x^{\bot_L}$ for the set of all complements to $x$ in $L$, or simply $x^\bot$ when it is clear what lattice we are using.
\end{definition}

For example in $L_{12}$, the lattice of divisors of 12, 3 is a complement for 4 but not for 6.

Theorem \ref{crapo}, often referred to as Crapo's Complement Theorem \cite{Crapo}, is a very powerful tool in
obtaining the M\"{o}bius number of a lattice.

\begin{theorem} \label{crapo} \cite{Crapo}
\label{crapo} Let $L$ be a lattice with minimum element 0 and maximum element 1.  Define $\zeta$ on $L\times L$ as the characteristic function of the relation $\leq$. Let $x\in L$ and let $x^\bot$ be the set of complements to $x$ in $L$.  Then $$ \mu(L)=\underset{y,z \in x^\bot}{\sum{}}\mu(0,y)\zeta(y,z)\mu(z,1).$$

In particular, if there exists $x\in L$ such that $x$ has no complement in $L,$ then $\mu \left( L\right) =0.$
\end{theorem}

Rephrasing the definition of a complement in group-theoretic terms, we have that if $H$ and $K$ are subgroups of $G$ then $H$ is a complement for $K$ in the lattice of subgroups of $G$ if and only if $H \cap K$ is the trivial group and $<H,K>=G$.  The following simple but useful lemma shows why typically in applying Crapo's Complement Theorem we want to look at complements to a normal subgroup.

\begin{lemma} \label{normalComplements}
Let $H\lhd G$.  Let $K_1$ and $K_2$ be two different complements to $H$ in $G$.  Then $K_1 \not \leq K_2$.
\end{lemma}
\begin{proof}
By the Second Isomorphism Theorem, any complement to $H$ must be isomorphic to $G/H$.  In particular, $|K_1|=|G/H|=|K_2|$.  Thus if $K_1$ and $K_2$ are distinct they are not comparable since if one was a subgroup of the other, they would be equal. 
\end{proof}

Equivalently rephrased in lattice-theoretic terms, a normal subgroup is always a {\bf modular element} of the subgroup lattice.  Thus in this case we can use a much simpler form of Crapo's Complement Theorem:

\begin{corollary}
\label{crapoMod} Let $H\lhd G$.  Let $H^\bot$ be the set of complements to $H$ in $G$.  Then $$ \mu(G)=\underset{K \in H^\bot}{\sum{}}\mu(K)\mu(K,G).$$
\end{corollary}
\begin{proof}
Lemma \ref{normalComplements} implies we have no distinct comparable complements to $H$.  Thus when picking pairs of complements to form the sum in Theorem \ref{crapo}, $\zeta$ vanishes unless the same complement is used twice.
\end{proof}

\subsection{Direct Products}

It would be very useful to have a formula for the M\"{o}bius number of a direct product of groups in terms of the M\"{o}bius numbers of the factors.  At the moment this problem seems intractable due to the incredibly large number of subdirect products which may arise.  However with restrictions on the factors, we get such a formula sufficiently general to account for all socles.  We begin by classifying complements to one factor in a direct product.  (We assume in this section the reader has a basic knowledge of the structure of direct and subdirect products of groups; see \cite{Dixon} for a detailed introduction if needed.)

\begin{lemma}  \label{comps}
Let $G=H \times K$.  Let $\phi_K:G \rightarrow K$ be the projection map onto $K$.  Let $K'$ be a complement for $H$ in the subgroup lattice of $G$.  Then $K'$ is of the form $K'=H_0 \subdir K$ for some $H_0\leq H$ with $H_0$ isomorphic to some homomorphic image of $K$.  Additionally, $\phi_K$ has trivial kernel when restricted to $K'$.
\end{lemma}
\begin{proof}
Let $\phi_H:G \rightarrow H$ similarly be the projection map onto $H$.  Let $K'\leq G$ and assume $K'$ is a complement of $H$.  Since $<H,K'>=G$ and $H=\ker{\phi_K}$, we have $\phi_K(K')=\phi_K(<H,K'>)=\phi_K(G)=K$.  Let $H_0=\phi_H(K')$.  By the Second Isomorphism Theorem, $K'$ is isomorphic to $K$.  Thus $H_0$ is isomorphic to a homomorphic image of $K$ since projection maps are homomorphisms.  The image lies in $H$, so $H_0 \leq H$.

Thus we have that $K'$ is a subdirect product of $H_0$ with $K'$.

Since $\ker(\phi_K)=H$, we have $\ker(\phi_K|_{K'})=K' \cap \ker(\phi_K)= K' \cap H$ which is trivial.  Thus we see that $\phi_K$ has trivial kernel when restricted to $K'$.
\end{proof}

Note the condition that $\phi_K$ has trivial kernel when restricted to $K'$ is intuitively saying that $H_0$ must be `fully glued' to some factor of $K'$.  This makes sense because having any `unglued' pieces of $H_0$ would intersect nontrivially with $H$ and thus $K'$ would not be a complement.

Also notice that this does in fact give an algorithm for constructing all complements to $H$ in $G=H\times K$.  We get one class of complements for each such $H_0$ (an automorphism may be applied to form the subdirect product in a different manner). Thus finding all complements simply amounts to finding all subgroups of $H$ isomorphic to a homomorphic image of $K$ and knowing the automorphisms of these subgroups.  We use this strategy to compute the M\"{o}bius numbers of some special direct products.

\begin{lemma} \label{muMult}
Let $G=H \times K$ and assume $K$ has no nontrivial homomorphic image isomorphic to a subgroup of $H$.  Then $\mu(H \times K)=\mu(H)\mu(K)$.
\end{lemma}
\begin{proof}
We look at complements to $H$ in the subgroup lattice of $G$.  The only possibility for $H_0$ as described in Lemma \ref{comps} is the trivial group, which gives $K$ itself as the only complement.  Since $K$ is normal in $G$, the interval of subgroups between $K$ and $G$ is isomorphic to the subgroup lattice of $H$ by the Fourth Isomorphism Theorem.  Thus we can apply Crapo's Complement Theorem in the simpler form given in Corollary \ref{crapoMod} to get that the M\"{o}bius number of $G$ is $$\mu(G)= \mu(K)\mu(K,G)=\mu(K)\mu(H).$$
\end{proof}

\begin{example}
One can easily compute (for example just from the definition of $\mu$ using \texttt{GAP}) that the M\"{o}bius number of $S_3$ is 3 and the M\"{o}bius number of $A_5$ is -60.  Thus the M\"{o}bius number of $S_3 \times A_5$ is -180 since the only homomorphic images of $A_5$ are trivial or $A_5$ itself, and $A_5$ is not isomorphic to a subgroup of $S_3$.
\end{example}

Notice that Lemma \ref{muMult} has a very nice aesthetic parallel to a similar property of the number-theoretic M\"{o}bius function.  Namely, for any integers $n$ and $m$, if $n$ and $m$ are coprime then $\mu(nm)=\mu(n)\mu(m)$.  Here we have a similar result for groups; it would seem that the hypothesis of Lemma \ref{muMult} is a sufficient condition for whatever it means for two groups to be coprime!

Also notice that Lemma \ref{muMult} does not follow from the fact that the M\"{o}bius number of a product poset is the product of the M\"{o}bius numbers of each factor.  In general the lattice of subgroups of $H \times K$ will be much larger than the product lattice of the subgroups of $H$ with the subgroups of $K$.  The example above illustrates this: many subdirect products of $S_2$ (a homomorphic image of $S_3$) with a subgroup of $A_5$ isomorphic to $S_2$ will appear in the subgroup lattice of $S_3 \times A_5$.  These subdirect products are not direct products of subgroups of each lattice.

\subsection{The Homomorphic Images of a Product of Simple Groups}

In order to apply Lemma \ref{muMult}, we will first want a classification of all possible homomorphic images of a direct product of simple groups.

\begin{lemma} \label{homIms}
Let $U_1, U_2, \ldots, U_m$ be distinct nonabelian simple groups. Let $G=U_1^{e_1}\times U_2^{e_2} \times \cdots \times U_m^{e_m}$.  Then every homomorphic image of $G$ is of the form $U_1^{f_1}\times U_2^{f_2} \times \cdots \times U_m^{f_m}$ for some natural numbers $f_i \leq e_i$ for $i \in \{1,2,\ldots,m\}$.

\end{lemma}
\begin{proof}  This result is equivalent to showing that any normal subgroup of $G$ also has the above form.  
Let $N \lhd G$.  Let $\pi$ be the projection of $N$ onto $U$, some factor of $G$.  Since $U$ is simple, $\pi(N)$ is either trivial or all of $U$, since intersecting a normal subgroup of $G$ with a factor of $G$ will produce again a normal subgroup of $G$.  Thus since $N$ projects trivially or fully onto each factor, $N$ is a subdirect product powers of the $U_i$.  Simplicity of the $U_i$ implies that any subdirect product actually degenerates to a direct product, so $N$ is of the required form.
\end{proof}

Note that additionally for the normal subgroup, the `gluing' in the subdirect product must be trivial, ie we cannot have any diagonal subdirect products of the simple groups in $N$.  If we did, one could apply an inner automorphism, conjugating by an element of one of the simple groups in the diagonal subdirect product, and $N$ would not be normal.

\subsection{The M\"{o}bius Number of a Direct Power of an Abelian Simple Group}

A direct power of an abelian simple group looks like $C_p^n=\underset{n}{\underbrace{C_p \times C_p \times \cdots \times C_p}}$ for some prime $p$ and some integer $n$.  This is clearly isomorphic to a vector space of dimension $n$ over the field with $p$ elements, and the subgroups of such a group will correspond to the subspaces of a such a vector space.  The M\"{o}bius number of a finite vector space is well-known and can be found for example in Chapter 3 Exercises 28 and 45 of \cite{Stanley}.  Thus we have:

\begin{corollary}\label{mobVec}
The M\"{o}bius number of $C_p^n$ is $(-1)^np^{\binom{n}{2}}$.
\end{corollary}

\subsection{The M\"{o}bius Number of a Direct Power of a Nonabelian Simple Group}

\begin{theorem} \label{dirProd}
Let $T$ be a nonabelian simple group.  Let $T^n=\underset{n}{\underbrace{T \times T \times \cdots \times T}}$.  Then $$\mu(T^n)=\overset{n}{\underset{{j=1}}{\prod{}}}{(\mu(T)-(j-1)A)}$$ where $A$ is the size of the automorphism group of $T$.
\end{theorem}
\begin{proof}
We use Lemma \ref{comps} to enumerate all possible complements to $T_1$.  In each case we analyze the structure of the interval of the subgroup lattice lying above the complement.  Finally we apply Crapo's Complement Theorem.  

Again let $K'$ be a complement to $T_1$.  Let $K=T_2 \times \cdots \times T_n$.  Then as in the lemma, $K'=H_0 \subdir K$ for some $H_0 \leq T_1$ with $H_0$ a homomorphic image of $K$.  By Lemma \ref{homIms}, the only homomorphic images of $K$ are isomorphic to $1,T,T^2,T^3,\cdots,T^{n-1}$.  However, only $1$ and $T$ are subgroups of $T_1$, so they are the only two choices for $H_0$.

Case 1: If $H_0\cong 1$, $K'=K$.  In this case the interval $[K,G]$ is isomorphic to the subgroup lattice of $T$ since $G/K \cong T_1$ by the First Isomorphism Theorem.

Case 2: If $H_0\cong T$, $K'$ is a subdirect product of $T_1$ and $T_2 \times \cdots \times T_n$.  In this case we claim the interval $[K',G]$ is just a totally ordered lattice with two nodes.  That is, $K'$ is maximal in $G$.  To see this is true, assume we have a subgroup $J$ with $K'\leq J \leq G$.  Then any projection of $K'$ has to be a subgroup of the same projection of $J$.  That is, $\phi_{T_1}(K')=T_1$ and $\phi_{K}(K')=K$ implies $\phi_{T_1}(J)=T_1$ and $\phi_{K}(J)=K$ as well.  Thus $J$ is also a subdirect product of $T_1$ and $K$.  Since $\ker(\phi_K|_J)$ must be a normal subgroup of $T_1$, it is either $1$ or $T_1$.  If $\ker(\phi_K|_J)=1$ then $J=K'$.  If $\ker(\phi_K|_J)=T_1$ then $J=G$.  Thus $K'$ is maximal in $G$.  The M\"{o}bius number of such a lattice, a totally ordered lattice with two elements, is -1.

Now we just have to count how many different ways each can occur.  Clearly there is only one way Case 1 can happen.  However, for Case 2, in the subdirect product $H_0$ can be identified with any of the $n-1$ different homomorphic images of $K$ isomorphic to $T$.  Additionally, any automorphism of $T$ can be applied to $H_0$ to identify it in a different way.  Thus we have $(n-1)A$ different subdirect products in Case 2 where $A=|Aut(T)|$.

At last we apply Crapo's Complement Theorem, summing over complements of $T_1$.  This yields 
$$\mu(T^n)=\mu(T)\mu(T^{n-1})+(n-1)A\mu(T)(-1) $$  
  
which can be viewed as a recurrence relation with respect to $n$.  Solving the recurrence proves the theorem.
  
\end{proof}

\subsection{Proof of Theorem \ref{main}}

Combining Theorem \ref{dirProd} with Lemma \ref{muMult}, we get the M\"{o}bius number of any socle in terms of the M\"{o}bius numbers of the socle types.

\begin{proof}
First assume $G$ has no abelian direct factors.  Then $G=U_1^{e_1} \times U_2^{e_2} \times \cdots \times U_m^{e_m}$ for some nonabelian simple groups $U_1, U_2, \ldots U_m$.  Without loss of generality, assume $|U_1|<|U_2|<\ldots<|U_m|$.  We claim that taking $H=U_1^{e_1}$ and $K=U_2^{e_2} \times \cdots \times U_m^{e_m}$ satisfies the conditions of Lemma \ref{muMult}.  By Lemma \ref{homIms}, we have that the only homomorphic images of $K$ are of the form $U_2^{f_2} \times \cdots \times U_m^{f_m}$ for some $f_i \leq e_i$.  Assume such a group was isomorphic to a subgroup of $H$.  Then we have $U_2 \leq U_2^{f_2} \times \cdots \times U_m^{f_m} \leq H$ (without loss of generality $f_2>0$).  Consider the projection maps from $H$ onto its factors (the copies of $U_1$).  At least one of these projection maps must have a nontrivial image when applied to $U_2 \leq H$, otherwise $U_2$ would be trivial.  Let $\pi:U_2\rightarrow U_1$ be such a map with nontrivial image.  Since $\ker(\pi)$ is a normal subgroup of $U_2$ and the image under $\pi$ is nontrivial, the kernel must be trivial and $\pi(U_2)=U_2 \leq U_1$.  But we cannot have $U_2 \leq U_1$ since $U_2$ has larger order than $U_1$.
Thus applying Lemma \ref{muMult} gives $\mu(G)=\mu(U_1^{e_1})\mu(U_2^{e_2} \times \cdots \times U_m^{e_m})$.  By repeatedly applying this process next with $U_2$, then with $U_3$, and so on, we can write $$\mu(G)=\mu(U_1^{e_1})\mu(U_2^{e_2}) \cdots \mu(U_m^{e_m})$$ and combining the above formula with Theorem \ref{dirProd} proves the result for the case where $G$ has no abelian direct factors.  We now must show that any abelian direct factors split off multiplicatively as well.  

Once again let $G=C_{p_1}^{f_1}\times C_{p_2}^{f_2} \times \cdots \times C_{p_n}^{f_n} \times U_1^{e_1}\times U_2^{e_2} \times \cdots \times U_m^{e_m}$.  Taking $H=C_{p_1}^{f_1}\times C_{p_2}^{f_2} \times \cdots \times C_{p_n}^{f_n}$ and $K=U_1^{e_1}\times U_2^{e_2} \times \cdots \times U_m^{e_m}$ satisfies the conditions of Lemma \ref{muMult}, since an abelian group could not possibly have a subgroup isomorphic to a homomorphic image of a direct product of nonabelian simple groups (since Lemma \ref{homIms} shows that any such homomorphic image is nonabelian).  Thus $\mu(G)=\mu(H)\mu(K)$, so the abelian part does indeed split off multiplicatively.

At last we compute the M\"{o}bius number of the abelian part.   Again Lemma \ref{muMult} implies that $$\mu\left(C_{p_1}^{f_1}\times C_{p_2}^{f_2} \times \cdots \times C_{p_n}^{f_n}\right)=\mu\left(C_{p_1}^{f_1}\right)\mu\left(C_{p_2}^{f_2}\right)\cdots \mu\left(C_{p_n}^{f_n}\right)$$ since no power of a cyclic group $C_p$ could have a subgroup isomorphic to the homomorphic image of an abelian group that does not have elements of order $p$.  Applying Corollary \ref{mobVec} to each factor proves the result.

\end{proof}

\begin{example}
The M\"{o}bius number of $C_2$ is -1.  The M\"{o}bius number of $A_5$ is -60 and its automorphism group has size 120.  The M\"{o}bius number of $A_6$ is 720 and its automorphism group has size 1440.  Thus we have

$\mu(C_2 \times A_5^3 \times A_6^2)=-1*-60*(-60-120)*(-60-2*120)*720*(720-1440)=-1,679,616,000,000$.

This is a computation that would clearly not be feasible by any brute force enumeration of the subgroup lattice!
\end{example}

It should be noted that there is an ongoing project to compute the M\"{o}bius numbers of all sporadic simple groups that has computed the M\"{o}bius numbers of Ru and Suz \cite{bohanon}.  Additionally \cite{ShareshianThesis} has the M\"{o}bius numbers of some infinite families of simple groups, including linear groups of dimension two.

\pagebreak

\bibliography{PrelimBib}{}
\bibliographystyle{plain}

\end{document}